\documentclass[11pt,reqno]{amsart}

\usepackage{amsmath,amssymb,mathrsfs}
\usepackage{graphicx,cite}

\newtheorem{theo}{Theorem}[section]

\newtheorem{lemm}[theo]{Lemma}

\newtheorem{rema}[theo]{Remark}

\numberwithin{equation}{section}

\begin{document}

\title[stability on inverse source problem]{Stability on the
one-dimensional inverse source scattering problem in a two-layered medium}

\author{Yue Zhao}
\address{School of Mathematics and Statistics, Jiangsu Normal University,
Xuzhou, Jiangsu 221116, China.}
\email{zhaoy@jsnu.edu.cn}

\author{Peijun Li}
\address{Department of Mathematics, Purdue University, West Lafayette, Indiana
47907, USA.}
\email{lipeijun@math.purdue.edu}

\thanks{The research of PL was supported in part by the NSF grant DMS-1151308.}



\begin{abstract}
This paper concerns the stability on the inverse source scattering problem for
the one-dimensional Helmholtz equation in a two-layered medium. We show that the
increasing stability can be achieved by using multi-frequency wave field at
the two end points of the interval which contains the compact support of the
source function. 
\end{abstract}

\maketitle

\section{Introduction and problem formulation}

We consider the one-dimensional Helmholtz equation in a two-layered medium:
\begin{align}\label{ode}
u^{\prime \prime}(x, \omega) + \kappa^2(x)u(x,\omega) = f(x),\quad x\in(-1, 1),
\end{align}
where $\omega>0$ is the angular frequency, the source function $f$ has a compact
support which is assumed to be contained in the interval $(-1,1)$, and the wave
number $\kappa$ satisfies
\[
\kappa(x) = \begin{cases}
\kappa_1, \quad x>0,\\
\kappa_2, \quad x<0.
\end{cases}
\]
Here $\kappa_j = c_j\omega, j=1, 2,$ where $c_j>0$ are constants. The
wave field $u$ is required to satisfy the outgoing wave conditions:
\begin{equation}\label{orc}
u^{\prime}(-1,\omega) + {\rm i}\kappa_2 u(-1,\omega) = 0, \quad
u^{\prime}(1,\omega) - {\rm i}\kappa_1 u(1,\omega) = 0.
\end{equation}

Given $f\in L^2(-1,1)$, it is known that the problem \eqref{ode}--\eqref{orc}
has a unique solution:
\begin{align}\label{sol}
u(x, \omega) = \int_0^1 g(x,y)f(y){\rm d}y,
\end{align}
where $g$ is the Green function given as follows
\[
 g(x, y)=\begin{cases}
           {\rm i} \frac{\kappa_1 - \kappa_2}{2\kappa_1(\kappa_1 +
\kappa_2)}e^{{\rm i}\kappa_1 (x+y)} + \frac{\rm i}{2\kappa_1}e^{{\rm i}\kappa_1
|x-y|}, & x>0,\\
 \frac{\rm i}{\kappa_1 + \kappa_2}e^{{\rm i}(\kappa_1 y-\kappa_2 x)}, & x<0,
         \end{cases}\quad\text{for} ~ y>0,
\]
and
\[
 g(x,y)=\begin{cases}
  {\rm i} \frac{\kappa_2 - \kappa_1}{2\kappa_2(\kappa_1 + \kappa_2)}e^{-{\rm
i}\kappa_2 (x+y)} + \frac{\rm i}{2\kappa_2}e^{{\rm i}\kappa_2 |x-y|}, &
x<0,\\[10pt]
 \frac{\rm i}{\kappa_1 + \kappa_2}e^{{\rm i}(- \kappa_2 y + \kappa_1 x )}, &
x>0,
 \end{cases}\quad\text{for} ~ y<0.
\]

This paper concerns the inverse source problem: Let $f$ be a
complex function with a compact support contained in $(-1, 1)$. The inverse
problem is to determine $f$ by using the boundary data $u(-1, \omega)$ and
$u(1, \omega)$ with $\omega\in (0, K)$ where $K >1$ is a positive constant.

The inverse source scattering problem have significant applications in antenna
synthesis, medical imaging, and optical tomography \cite{A-IP99, I-89}. They
been extensively investigated by many researchers \cite{ACTV-IP12, AM-IP06,
ABF-SIAP02, BN-IP11, DML-SIAP07, EV-IP09}. It is known that there is no
uniqueness for the inverse source problems at a fixed frequency due to the
existence of non-radiating sources \cite{DS-IEEE82, HKP-IP05}. Recently, it has
been realized that the use of multi-frequency data can not only overcome the
difficulties of non-uniqueness, which are presented at a single frequency, but
also achieve increasing stability \cite{BLLT-IP15, BLT-JDE10, CIL-JDE16,
LY-IPI17, LY-JMAA17}. These work assume that the medium is homogeneous in
the whole space. In this work, we intend to establish the increasing stability
on the inverse source problem for the one-dimensional Helmholtz equation in a
two-layered medium.

\section{Main result}

Define a functional space:
\[
 \mathcal{F}_M =\{f\in H^n(-1, 1): \|f\|_{H^n(-1, 1)}\leq M, ~{\rm
supp}f\subset(-1, 1)\},
\]
where $n\in \mathbb N$ and $M>1$ is a constant. Hereafter, the notation
"$a\lesssim b$" stands for $a\leq Cb$, where $C$ is a generic
constant independent of $n, \omega, K, M,$ but may change step by step in the
proofs.

The following stability estimate is the main result of this paper.

\begin{theo}\label{mr1}
Let $f\in\mathcal{F}_M$ and let $u$ be the solution \eqref{sol} corresponding to
$f$. Then we have
\begin{align}
\label{cfe} \| f\|^2_{L^2(-1, 1)}\lesssim
\epsilon^2+\frac{M^2}{\left(\frac{K^{\frac{2}{3}}|\ln
\epsilon|^{\frac{1}{4}}}{(6n-3)^3}\right)^{ 2n-1}},
\end{align}
where 
\begin{align}
\label{ep} \epsilon&=\left(\int_0^K \omega^2\left( |u(-1, \omega)|^2 +
|u(1, \omega)|^2 \right){\rm d}\omega\right)^{\frac{1}{2}}.
\end{align}
\end{theo}

\begin{rema}
The stability estimate \eqref{cfe} consists of two parts: the data discrepancy
and the high frequency tail. The former is of the Lipschitz type. The latter
decreases as K increases which makes the problem have an almost Lipschitz
stability. The result explains that the problem becomes more stable when higher
frequency data is used. The stability estimate \eqref{cfe} also implies the
uniqueness of the inverse source problem.
\end{rema}

\section{Proof of Theorem \ref{mr1}}

Consider the following two functions:
\begin{equation}\label{f12}
f_1(x) = \begin{cases}
f(x), \quad &x>0,\\
0, \quad &x<0,
\end{cases}
\quad\text{and}\quad
f_2(x) = \begin{cases}
0, \quad &x>0,\\
f(x), \quad &x<0.
\end{cases}
\end{equation}

\begin{lemm}\label{3.1}
Let $f \in L^2(-1,1)$ with ${\rm supp}f \subset (-1,1).$ We have
\[
\|f\|^2_{L^2(-1,1)}\lesssim \int_{0}^{+\infty} \omega^2 \left(|u(-1,\omega)|^2 +
|u(1,\omega)|^2\right){\rm d}\omega.
\]
\end{lemm}

\begin{proof}
Choosing $\xi_1 \in \mathbb R$ with $|\xi_1| = \kappa_1$, multiplying both sides
of \eqref{ode} by $e^{-{\rm i}\xi_1 x}$, and integrating over $(0,1)$ with
respect to $x$, we have from the integration by parts that 
\begin{equation}\label{ibp1}
e^{-{\rm i}\xi_1}u^{\prime}(1, \omega) + {\rm i}\xi_1 e^{-{\rm i}\xi_1}
u(1,\omega) - u^{\prime}(0,\omega)- {\rm i}\xi_1 u(0,\omega) = \int_0^1 e^{-{\rm
i}\xi_1 x} f_1(x){\rm d}x.
\end{equation}
Similarly, choosing $\xi_2 \in \mathbb R$ with $|\xi_2| = \kappa_2$, multiplying
both sides of \eqref{ode} by $e^{-{\rm i}\xi_2 x}$, and integrating over
$(-1,0)$ with respect to $x$, we have from the integration by parts that 
\begin{equation}\label{ibp2}
-e^{-{\rm i}\xi_2}u^{\prime}(-1, \omega) - {\rm i}\xi_2 e^{{\rm i}\xi_2}
u(-1,\omega) + u^{\prime}(0,\omega)+ {\rm i}\xi_2 u(0,\omega) = \int_{-1}^0
e^{-{\rm i}\xi_2 x} f_2(x){\rm d}x.
\end{equation}
It follows from \eqref{sol} and \eqref{f12} that
\[
u(x,\omega) = \int_0^1 g(x, y)f(y){\rm d}y= \int_0^1 g(x,y) f_1(y){\rm d}y +
\int_{-1}^0 g(x,y) f_2(y){\rm d}y,
\]
which gives
\begin{align}\label{u0}
u(0,\omega) &= \int_0^1 g(0,y) f_1(y){\rm d}y + \int_{-1}^0 g(0,y) f_2(y){\rm
d}y\notag\\
& = \int_0^1 \frac{\rm i}{\kappa_1 + \kappa_2}e^{{\rm i}\kappa_1 y}f_1(y){\rm
d}y + \int_{-1}^0 \frac{\rm i}{\kappa_1 + \kappa_2}e^{-{\rm i}\kappa_2
y}f_2(y){\rm d}y.
\end{align}
On the other hand, we have from a simple calculation that
\begin{align}\label{du0}
u^{\prime}(0,\omega) &= \int_0^1 g^{\prime}(0,y) f_1(y){\rm d}y + \int_{-1}^0
g^{\prime}(0,y) f_2(y){\rm d}y\notag\\
& = \int_0^1 \frac{\kappa_2}{\kappa_1 + \kappa_2}e^{{\rm i}\kappa_1 y}f_1(y){\rm
d}y + \int_{-1}^0 \frac{-\kappa_1}{\kappa_1 + \kappa_2}e^{-{\rm i}\kappa_2
y}f_2(y){\rm d}y.
\end{align}

Letting $\xi_1 = -\kappa_1$, we have from \eqref{u0} and \eqref{du0} that 
\begin{equation}\label{u1}
u^{\prime}(0,\omega)- {\rm i} \kappa_1 u(0,\omega) = \int_0^1 e^{{\rm i}\kappa_1
y}f_1(y){\rm d}y.
\end{equation}
Combining \eqref{u1} and \eqref{ibp1}, we obtain
\[
e^{{\rm i}\kappa_1}u^{\prime}(1, \omega) - {\rm i}\kappa_1 e^{{\rm
i}\kappa_1} u(1,\omega)  = 2\int_0^1 e^{{\rm i}\kappa_1 x} f_1(x){\rm d}x,
\]
Using the outgoing radiation condition \eqref{orc}, we get from the above
equation that 
\[
e^{{\rm i}\kappa_1}{\rm i}\kappa_1 u(1, \omega) - {\rm i}\kappa_1 e^{{\rm
i}\kappa_1} u(1,\omega)  = 2\int_0^1 e^{{\rm i}\kappa_1 x} f_1(x){\rm d}x,
\]
which implies
\begin{align}\label{1}
|\hat{f}_1(-\kappa_1)|^2 \lesssim \omega^2 |u(\omega, 1)|^2.
\end{align}

Letting $\xi_2 = \kappa_2$, we have from \eqref{u0} and \eqref{du0} that 
\begin{equation}\label{u2}
u^{\prime}(0,\omega) + {\rm i} \kappa_2 u(0,\omega) = -\int_{-1}^0 e^{-{\rm
i}\kappa_2 y}f_2(y){\rm d}y,
\end{equation}
Combining \eqref{u2}, \eqref{ibp2}, and \eqref{orc}, we obtain
\[
e^{-{\rm i}\kappa_2}{\rm i}\kappa_2u^{\prime}(-1, \omega) - {\rm i}\kappa_2
e^{{\rm i}\kappa_2} u(-1,\omega)  = 2\int_{-1}^0 e^{-{\rm i}\kappa_2 x}
f_2(x){\rm d}x,
\]
which shows
\begin{align}\label{2}
|\hat{f}_2(\kappa_2)|^2 \lesssim \omega^2 |u(\omega, -1)|^2.
\end{align}

Letting $\xi_1 = \kappa_1$, we get from \eqref{u0} and \eqref{du0} that 
\begin{align}\label{u3}
u^{\prime}(0,\omega) + {\rm i} \kappa_1 u(0,\omega)& = \int_0^1 \frac{\kappa_2 -
\kappa_1}{\kappa_1 + \kappa_2}e^{{\rm i}\kappa_1 y} f_1(y){\rm d}y - \int_{-1}^0
\frac{2\kappa_1}{\kappa_1 + \kappa_2}e^{-{\rm i}\kappa_2 y} f_2(y){\rm
d}y\notag\\
&= \frac{\kappa_2 - \kappa_1}{\kappa_1 + \kappa_2} \hat{f}_1 (-\kappa_1) -
\frac{2\kappa_1}{\kappa_1 + \kappa_2}\hat{f}_2 (\kappa_2).
\end{align}
It follows from \eqref{u3}, \eqref{ibp1}, and \eqref{orc} that we obtain
\begin{align*}
e^{-{\rm i}\kappa_1}{\rm i}\kappa_1 u(1, \omega) + {\rm i}\kappa_1 e^{-{\rm
i}\kappa_1} u(1,\omega)  -\frac{\kappa_2 - \kappa_1}{\kappa_1 + \kappa_2}
\hat{f}_1 (-\kappa_1) + \frac{2\kappa_1}{\kappa_1 + \kappa_2}\hat{f}_2
(\kappa_2)= \hat{f}_1 (\kappa_1),
\end{align*}
which means
\begin{align}\label{3}
|\hat{f}_1 (\kappa_1)|^2 \lesssim \omega^2|u(1, \omega)| + |\hat{f}_1 (-\kappa_1)|^2 + |\hat{f}_2 (\kappa_2)|^2.
\end{align}

Finally, letting $\xi_2 = -\kappa_2$, we have from \eqref{u0} and \eqref{du0}
that 
\begin{align}\label{u4}
u^{\prime}(0,\omega) - {\rm i} \kappa_2 u(0,\omega)& = \int_0^1 \frac{2\kappa_2
}{\kappa_1 + \kappa_2}e^{{\rm i}\kappa_1 y} f_1(y){\rm d}y + \int_{-1}^0
\frac{\kappa_2 - \kappa_1}{\kappa_1 + \kappa_2}e^{-{\rm i}\kappa_2 y} f_2(y){\rm
d}y\notag\\
&= \frac{2\kappa_2 }{\kappa_1 + \kappa_2} \hat{f}_1 (-\kappa_1) + \frac{\kappa_2
- \kappa_1}{\kappa_1 + \kappa_2}\hat{f}_2 (\kappa_2).
\end{align}
Using \eqref{u4}, \eqref{ibp2}, and \eqref{orc}, we have
\begin{align*}
e^{{\rm i}\kappa_2}{\rm i}\kappa_2 u(-1, \omega) + {\rm i}\kappa_2 e^{-{\rm
i}\kappa_2} u(-1,\omega)  +\frac{2\kappa_2 }{\kappa_1 + \kappa_2} \hat{f}_1
(-\kappa_1) + \frac{\kappa_2 - \kappa_1}{\kappa_1 + \kappa_2}\hat{f}_2
(\kappa_2) = \hat{f}_2 (-\kappa_2),
\end{align*}
which means
\begin{align}\label{4}
|\hat{f}_2 (-\kappa_2)|^2 \lesssim \omega^2 |u(-1,\omega)|^2 + |\hat{f}_1 (-\kappa_1)| + |\hat{f}_2 (\kappa_2)|.
\end{align}

Therefore, it follows from \eqref{1} that we get
\begin{align}\label{5}
\int_0^{\infty} |\hat{f}_1 (-\kappa_1)|^2{\rm d}\omega \lesssim \int_0^{\infty}
\omega^2 |u(\omega, 1)|^2{\rm d}\omega.
\end{align}
Using \eqref{2} gives 
\begin{align}\label{6}
\int_0^{\infty} |\hat{f}_2 (\kappa_2)|^2{\rm d}\omega \lesssim \int_0^{\infty}
\omega^2 |u(-1, \omega)|^2{\rm d}\omega.
\end{align}
It follows from \eqref{3}, \eqref{5}, and \eqref{6} that we have
\begin{align}\label{7}
\int_0^{\infty} |\hat{f}_1 (\kappa_1)|^2{\rm d}\omega \lesssim \int_0^{\infty}
\omega^2 |u(-1, \omega)|^2{\rm d}\omega + \int_0^{\infty} \omega^2 |u(1,
\omega)|^2{\rm d}\omega;
\end{align}
Finally following from \eqref{4}, \eqref{5}, and \eqref{6}, we obtain
\begin{align}\label{8}
\int_0^{\infty} |\hat{f}_2 (-\kappa_2)|^2{\rm d}\omega \lesssim \int_0^{\infty}
\omega^2 |u(-1, \omega)|^2{\rm d}\omega + \int_0^{\infty} \omega^2 |u(1,
\omega)|^2{\rm d}\omega.
\end{align}

We obtain from the Plancherel theorem that 
\begin{align*}
\|f\|_{L^2(-1,1)}^2 = \|f\|_{L^2(-\infty,\infty)}^2 =
\|\hat{f}\|_{L^2(-\infty,\infty)}^2 
\lesssim \|\hat{f}_1\|_{L^2(-\infty,\infty)}^2 +
\|\hat{f}_2\|_{L^2(-\infty,\infty)}^2.
\end{align*}
On the other hand, we have
\[
\|\hat{f}_1\|_{L^2(-\infty,\infty)}^2 = \int_0^{\infty}|\hat{f}_1(\omega)|^2
{\rm d}\omega + \int_0^{\infty}|\hat{f}_1(-\omega)|^2 {\rm d}\omega.
\]
Using \eqref{5} and \eqref{7} yields
\[
\|\hat{f}_1\|_{L^2(-\infty,\infty)}^2 \lesssim \int_{0}^{+\infty} \omega^2
\left(|u(-1,\omega)|^2 + |u(1,\omega)|^2\right){\rm d}\omega.
\]
Similarly, we have from \eqref{6} and \eqref{8} that
\[
\|\hat{f}_2\|_{L^2(-\infty,\infty)}^2 \lesssim \int_{0}^{+\infty} \omega^2
\left(|u(-1,\omega)|^2 + |u(1,\omega)|^2\right){\rm d}\omega.
\]
The proof is completed by combining the above estimates. 
\end{proof}

\begin{lemm}\label{3.2}
Let $f \in L^2(-1,1)$. We have
\begin{align*}
\omega^2|u(-1,\omega)|^2 &\lesssim \left| \int_0^1 e^{{\rm i}c_1 \omega
y}f_1(y){\rm d}y \right|^2 + \left| \int_{-1}^0 e^{-{\rm i}c_2
\omega y}f_2(y){\rm d}y \right|^2 + \left|\int_{-1}^0 e^{{\rm i}c_2 \omega
y}f_2(y){\rm d}y \right|^2,\\
\omega^2|u(1,\omega)|^2 &\lesssim \left| \int_0^1 e^{{\rm i}c_1 \omega
y}f_1(y){\rm d}y \right|^2 + \left| \int_0^1 e^{-{\rm i}c_1 \omega y}f_1(y){\rm
d}y \right|^2 + \left|\int_{-1}^0 e^{-{\rm i}c_2 \omega y}f_2(y){\rm d}y
\right|^2.
\end{align*}
\end{lemm}

\begin{proof}
It follows from \eqref{sol} that we have 
\begin{align*}
\omega u(-1,\omega) = & \int_0^1 \frac{\rm i}{c_1 + c_2}e^{{\rm i}(c_1 \omega y+
c_2 \omega)}f_1(y){\rm d}y + \int_{-1}^{0}  \frac{{\rm i}(c_2 - c_1)}{2c_2(c_1 +
c_2)}e^{-{\rm i}c_2 \omega(-1+y)}f_2(y){\rm d}y\\
&+ \int_{-1}^{0}\frac{\rm i}{2c_2}e^{-{\rm i}c_2 \omega(-1-y)}f_2(y){\rm d}y
\end{align*}
and
\begin{align*}
\omega u(1,\omega)  = &\int_0^1  \frac{{\rm i}(c_1 - c_2)}{2c_1(c_1 +
c_2)}e^{{\rm i}c_1 \omega(1+y)}f_1(y){\rm d}y + \int_0^1 \frac{\rm
i}{2c_1}e^{{\rm i}c_1\omega (1-y)} f_1(y) {\rm d}y\\
&+ \int_{-1}^0 \frac{\rm i}{c_1 + c_2}e^{{\rm i}(- c_2 \omega y + c_1\omega )}
f_2(y) {\rm d}y.
\end{align*}
The proof is done by taking square of the amplitudes on both sides of the
above equations.
\end{proof}

Next, let 
\[
I(s) = I_1(s) + I_2(s),
\]
where 
\[
I_1(s) = \omega^2 \int_0^s
|u(-1,\omega)|^2{\rm d}\omega, \quad I_2(s) = \omega^2 \int_0^s
|u(1,\omega)|^2{\rm d}\omega. 
\]
We have the following explicit representations for
$I_1(s)$ and $I_2(s)$:
\begin{align}\label{i1}
I_1(s) = &\int_0^s \bigg|\int_0^1 \frac{1}{c_1 + c_2}e^{{\rm i}(c_1 \omega y+
c_2 \omega)}f_1(y){\rm d}y + \int_{-1}^{0}  \frac{c_2 - c_1}{2c_2(c_1 +
c_2)}e^{-{\rm i}c_2 \omega(-1+y)}f_2(y){\rm d}y\notag\\
&+ \int_{-1}^{0}\frac{1}{2c_2}e^{-{\rm i}c_2 \omega(-1-y)}f_2(y){\rm d}y
\bigg|^2 {\rm d}\omega
\end{align}
and
\begin{align}\label{i2}
I_2(s) = & \int_0^s \bigg| \int_0^1  \frac{c_1 - c_2}{2c_1(c_1 + c_2)}e^{{\rm
i}c_1 \omega(1+y)}f_1(y){\rm d}y +
\int_0^1 \frac{1}{2c_1}e^{{\rm i}c_1\omega (1-y)} f_1(y) {\rm d}y \notag\\
& + \int_{-1}^0 \frac{1}{c_1 + c_2}e^{{\rm i}(- c_2 \omega y + c_1\omega )}
f_2(y) {\rm d}y\bigg|^2 {\rm d}\omega .
\end{align}

\begin{lemm}\label{3.3}
Let $f\in L^2(-1,1)$ and $c_{\rm max} = {\rm max}\{c_1, c_2\}$. We have for any
$s = s_1 + {\rm i}s_2, s_1, s_2\in \mathbb R$ that
\begin{align*}
|I_1(s)|&\lesssim |s|e^{4c_{\rm max}|s_2|}\int_0^1|f(y)|^2 {\rm d}y,\\
|I_2(s)|&\lesssim |s|e^{4c_{\rm max}|s_2|}\int_0^1|f(y)|^2 {\rm d}y.
\end{align*}
\end{lemm}

\begin{proof}
Let $\omega = st, t\in(0,1).$ A simple calculation yields
\begin{align*}
I_1(s) = & s \int_0^1 \bigg|\int_0^1 \frac{1}{c_1 + c_2}e^{{\rm i}(c_1 st y+ c_2
st)}f_1(y){\rm d}y +
\int_{-1}^{0}  \frac{c_2 - c_1}{2c_2(c_1 + c_2)}e^{-{\rm i}c_2 st(-1+y)}f_2(y){\rm d}y\notag\\
&+ \int_{-1}^{0}\frac{1}{2c_2}e^{-{\rm i}c_2 st(-1-y)}f_2(y){\rm d}y
\bigg|^2 {\rm d}t
\end{align*}
and
\begin{align*}
I_2(s) = & s\int_0^1 \bigg| \int_0^1  \frac{c_1 - c_2}{2c_1(c_1 + c_2)}e^{{\rm
i}c_1 st(1+y)}f_1(y){\rm d}y +
\int_0^1 \frac{1}{2c_1}e^{{\rm i}c_1 st (1-y)} f_1(y) {\rm d}y \notag\\
& + \int_{-1}^0 \frac{1}{c_1 + c_2}e^{{\rm i}(- c_2 st y + c_1 st )} f_2(y) {\rm
d}y\bigg|^2 {\rm d}t.
\end{align*}
Noting 
\[
|e^{\pm {\rm i}(c_1 st y+ c_2 st)}|\leq e^{2c_{\rm max}|s_2|}, \quad |e^{\pm
{\rm i}c_2 st(-1\pm y)}|\leq e^{2c_{\rm max}|s_2|}, \quad\forall y\in (-1,1), 
\]
we have from the Schwartz inequality that 
\begin{align*}
|I_1(s)|\lesssim |s|e^{4c_{\rm max}|s_2|}\int_{-1}^1|f(y)|^2 {\rm d}y.
\end{align*}
Similarly noting
\[
|e^{\pm {\rm i}c_1 st(1\pm y)}|\leq e^{2c_{\rm max}|s_2|}, \quad |e^{\pm
{\rm i}(- c_2 st y + c_1 st )}|\leq e^{2c_{\rm max}|s_2|}, \quad\forall y\in
(-1,1),
\] 
we get from the Schwartz inequality that
\begin{align*}
|I_2(s)|&\lesssim |s|e^{4c_{\rm max}|s_2|}\int_{-1}^1|f(y)|^2 {\rm d}y,
\end{align*}
which completes the proof. 
\end{proof}

\begin{lemm}\label{3.4}
Let $f \in H^n(-1,1), {\rm supp} f \subset (-1,1).$ We have for any $s>0$ that
\begin{align*}
\int_s^{\infty} \omega^2 (|u(-1,\omega)|^2 + |u(1,\omega)|^2){\rm
d}\omega\lesssim s^{-(2n-1)}\|f\|^2_{H^n(-1,1)}.
\end{align*}
\end{lemm}

\begin{proof}
It follows from Lemma \ref{3.2} that we have
\begin{align*}
&\int_s^{\infty} \omega^2|u(-1, \omega)|^2 {\rm d}\omega + \int_s^{\infty}
\omega^2|u(1, \omega)|^2 {\rm d}\omega\\ 
\lesssim &\int_s^{\infty} \left| \int_0^1 e^{{\rm i}c_1 \omega y}f_1(y){\rm
d}y\right|^2{\rm d}\omega + \int_s^{\infty} \left|\int_0^1 e^{-{\rm i}c_1 \omega
y}f_1(y){\rm d}y \right|^2{\rm d}\omega\\
& + \int_s^{\infty} \left| \int_{-1}^0 e^{{\rm i}c_2 \omega y}f_2(y){\rm
d}y\right|^2{\rm d}\omega + \int_s^{\infty} \left|\int_{-1}^0 e^{-{\rm i}c_2
\omega y}f_2(y){\rm d}y \right|^2{\rm d}\omega.
\end{align*}
Using the integration by parts and noting ${\rm supp}f_1 \subset (0,1)$ and
${\rm supp}f_2 \subset (-1,0)$, we obtain
\[
\int_0^1 e^{\pm{\rm i}c_1 \omega y}f_1(y){\rm d}y = \frac{1}{(\pm {\rm
i}c_1\omega)^n}\int_0^1 e^{\pm{\rm i}c_1 \omega y}f_1^{(n)}(y){\rm d}y
\]
and
\[
\int_{-1}^0 e^{\pm{\rm i}c_2 \omega y}f_2(y){\rm d}y = \frac{1}{(\pm {\rm
i}c_2\omega)^n}\int_{-1}^0 e^{\pm{\rm i}c_2 \omega y}f_2^{(n)}(y){\rm d}y,
\]
which give
\[
\left|\int_0^1 e^{\pm{\rm i}c_1 \omega y}f_1(y){\rm d}y\right|^2 \lesssim 
c_1^{-2n}\omega^{-2n}\|f_1\|^2_{H^n(0,1)}\lesssim
c_1^{-2n}\omega^{-2n}\|f\|^2_{H^n(-1,1)}
\]
and
\[
\left|\int_{-1}^0 e^{\pm{\rm i}c_2 \omega y}f_2(y){\rm d}y\right|^2 \lesssim 
c_2^{-2n}\omega^{-2n}\|f_2\|^2_{H^n(-1,0)}\lesssim
c_2^{-2n}\omega^{-2n}\|f\|^2_{H^n(-1,1)}.
\]
Hence we have
\begin{align*}
\int_s^{\infty}\left|\int_0^1 e^{\pm{\rm i}c_1 \omega y}f_1(y){\rm
d}y\right|^2{\rm d}\omega \lesssim 
c_1^{-2n}\|f_1\|^2_{H^n(0,1)}\int_s^{\infty}\omega^{-2n}{\rm d}\omega \lesssim
c_1^{-2n}\frac{s^{-(2n-1)}}{(2n-1)}\|f\|^2_{H^n(-1,1)}
\end{align*}
and
\begin{align*}
\int_s^{\infty}\left|\int_{-1}^0 e^{\pm{\rm i}c_2 \omega y}f_2(y){\rm
d}y\right|^2{\rm d}\omega \lesssim 
c_2^{-2n}\|f_2\|^2_{H^n(-1,0)}\int_s^{\infty}\omega^{-2n}{\rm d}\omega \lesssim
c_2^{-2n}\frac{s^{-(2n-1)}}{(2n-1)}\|f\|^2_{H^n(-1,1)},
\end{align*}
which completes the proof.
\end{proof}

The following lemma is proved in \cite{CIL-JDE16}. 

\begin{lemm}\label{3.5}
 Denote $S=\{z=x+{\rm i}y\in\mathbb{C}: -\frac{\pi}{4}<{\rm arg}
z<\frac{\pi}{4}\}$. Let $J(z)$ be analytic in $S$ and continuous in $\bar{S}$
satisfying
\[
 \begin{cases}
  |J(z)|\leq\epsilon, & z\in (0, ~ L],\\
  |J(z)|\leq V, & z\in S,\\
  |J(0)|=0.
 \end{cases}
\]
Then there exits a function $\mu(z)$ satisfying
\[
 \begin{cases}
  \mu(z)\geq\frac{1}{2},  & z\in(L, ~ 2^{\frac{1}{4}}L),\\
  \mu(z)\geq \frac{1}{\pi}((\frac{z}{L})^4-1)^{-\frac{1}{2}}, & z\in
(2^{\frac{1}{4}}L, ~ \infty)
 \end{cases}
\]
such that
\[
|J(z)|\leq V\epsilon^{\mu(z)}, \quad\forall\, z\in (L, ~ \infty).
\]
\end{lemm}

\begin{lemm}
 Let $f\in\mathcal{F}_M$. Then there exists a function
$\mu(s)$ satisfying
\begin{equation}\label{mu}
 \begin{cases}
  \mu(s)\geq\frac{1}{2}, \quad & s\in(K, ~ 2^{\frac{1}{4}}K),\\
  \mu(s)\geq \frac{1}{\pi}((\frac{s}{K})^4-1)^{-\frac{1}{2}}, \quad & s\in
(2^{\frac{1}{4}}K, ~\infty),
 \end{cases}
\end{equation}
such that
\[
 |I(s)|\lesssim M^2 e^{as}\epsilon^{2\mu(s)}, \quad\forall s\in (K, ~\infty),
\]
where $a = {\rm max}\{5c, 3\}$.
\end{lemm}

\begin{proof}
It follows from Lemma \ref{3.3} that
\[
|I_1(s)e^{-as}|\lesssim M^2, \quad |I_2(s)e^{-as}|\lesssim M^2, \quad
s\in S.
\]
Recalling \eqref{ep}, \eqref{i1}, and \eqref{i2}, we have
\[
|I_1(s)e^{-as}|\lesssim \epsilon^2, \quad |I_2(s)e^{-as}|\lesssim
\epsilon^2,\quad s\in[0, K].
\]
A direct application of Lemma \ref{3.5} shows that there exists a function $\mu(s)$ satisfying \eqref{mu} such that
\[
|I_1(s)e^{-as}| \lesssim M^2 \epsilon^{2\mu}, \quad |I_2(s)e^{-as}| \lesssim
M^2 \epsilon^{2\mu},\quad s\in (K, \infty),
\]
then we have
\[
|I(s)e^{-as}| = |I_1(s)e^{-as} + I_2(s)e^{-as}| \lesssim M^2 \epsilon^{2\mu},
\quad s\in (K, \infty),
\]
which completes the proof.
\end{proof}

Now we show the proof of Theorem \ref{mr1}.

\begin{proof}
We can assume that $\epsilon <e^{-1}$, otherwise the estimate is
obvious. Let
\[
s=\begin{cases}
  \frac{1}{(3\pi)^{\frac{1}{3}}}K^{\frac{2}{3}}|\ln\epsilon|^{\frac{1}{4}}, &
2^{\frac{1}{4}} (3\pi)^{\frac{1}{3}}K^{\frac{1}{3}}<|\ln\epsilon|^{\frac{1}{4}},\\
K, &|\ln\epsilon|\leq 2^{\frac{1}{4}}(3\pi)^{\frac{1}{3}}K^{\frac{1}{3}}.
 \end{cases}
\]
If $2^{\frac{1}{4}}(3\pi)^{\frac{1}{3}}K^{\frac{1}{3}}<|\ln\epsilon|^{\frac{1}{4}}$, then we have
\begin{align*}
 |I(s)|&\lesssim M^2 e^{as}
e^{-\frac{2|\ln\epsilon|}{\pi}((\frac{s}{K})^4-1)^{-\frac{1}{2}}}\lesssim M^2
e^{\frac{a}{(3\pi)^{\frac{1}{3}}}K^{\frac{2}{3}}|\ln\epsilon|^{\frac{1}{4}}-\frac{2|\ln\epsilon|}{\pi}
(\frac{K}{s})^2}\\
&=M^2
e^{-2\left(\frac{a^3}{3\pi}\right)^{\frac{1}{3}}K^{\frac{2}{3}}|\ln\epsilon|^{\frac{1}{2}}\left(1-\frac{1}{2}
|\ln\epsilon|^{-\frac{1}{4}}\right)}.
\end{align*}
Noting $\frac{1}{2} |\ln\epsilon|^{-\frac{1}{4}}<\frac{1}{2}$ and $a\geq 3$, we
have $\left(\frac{a^3}{3\pi}\right)^{\frac{1}{3}}\geq
\left(\frac{3^3}{3\pi}\right)^{\frac{1}{3}}>1$ and 
\[
 |I(s)| \lesssim M^2
e^{-K^{\frac{2}{3}}|\ln\epsilon|^{\frac{1}{2}}}.
\]
Using the elementary inequality
\[
 e^{-x}\leq \frac{(6n-3)!}{x^{3(2n-1)}}, \quad x>0,
\]
we get
\[
 |I(s)|\lesssim\frac{M^2}{\left(\frac{K^2|\ln\epsilon|^{\frac{3}{2}}}{(6n-3)^3}
\right)^{2n-1}}.
\]
If $|\ln\epsilon|\leq 2^{\frac{1}{4}}(3\pi)^{\frac{1}{3}}K^{\frac{1}{3}}$, then $s=K$. We have
from \eqref{ep} and Lemma \ref{3.2} that
\[
 |I(s)|\leq \epsilon^2.
\]
Hence we obtain from Lemma \ref{3.4} that
\begin{align*}
 &\int_0^\infty \omega^2 \left(|u(-1, \omega)|^2+|u(1, \omega)|^2\right) {\rm
d}\omega\\
&\lesssim \epsilon^2+\frac{M^2}{\left(\frac{K^2|\ln\epsilon|^{\frac{3}{2}}}{
(6n-3)^3} \right)^{2n-1}}+\frac{
\|f\|^2_{H^n(-1,
1)}}{\left(2^{-\frac{1}{4}}(3\pi)^{-\frac{1}{3}}K^{\frac{2}{3}}|\ln\epsilon|^{\frac{1}{4}}\right)^{2n-1}}.
\end{align*}
By Lemma \ref{3.1}, we have
\[
 \|f\|^2_{L^2(-1, 1)}\lesssim \epsilon^2
+\frac{M^2}{\left(\frac{K^2|\ln\epsilon|^{\frac{3}{2}}}{(6n-3)^3}\right)^{2n-1}}+\frac{M^2}{\left(\frac{K^{\frac{2}
{3}}|\ln\epsilon|^{\frac{1}{4}}}{(6n-3)^3}\right)^{2n-1}}.
\]
Since $K^{\frac{2}{3}}|\ln\epsilon|^{\frac{1}{4}}\leq K^2
|\ln\epsilon|^{\frac{3}{2}}$ when $K>1$ and $|\ln\epsilon|>1$, we obtain
the stability estimate.
\end{proof}

\section{Conclusion}

In this paper, we show that the increasing stability can be obtained for
the inverse source scattering problem of the one-dimensional Helmholtz equation
in a two-layered medium by using multi-frequency Dirichlet data
at the two end points of an interval which contains the compact support of
the source. The stability estimate consists of the data discrepancy and the high
frequency tail of the source function. We believe that the proposed method can
be extended to handle a multi-layered medium. Another possible future work is
to investigate the higher dimensional problem.

\appendix

\section{Green's function in a two-layered medium}

Consider the equation
\begin{equation}\label{gfe}
\frac{{\rm d^2}g(x,y)}{{\rm d}x^2} + \kappa^2(x) g(x,y) = - \delta(x-y),
\end{equation}
where $\delta$ is the Dirac delta function and the wavenumber $\kappa$ is a
piecewise constant, i.e., 
\[
 \kappa(x) = \begin{cases}
  \kappa_1, & x>0,\\
  \kappa_2, & x<0.\\
 \end{cases}
\]

If $y>0$, the solution of \eqref{gfe} has the following form
\[
 g(x,y) = \begin{cases}
  Ae^{{\rm i}\kappa_1 x}, & x>y,\\
  Be^{{\rm i}\kappa_1 x} + Ce^{-{\rm i}\kappa_1 x} , & 0<x<y,\\
  De^{-{\rm i}\kappa_2 x}, & x<0,
 \end{cases}
\]
where $A, B, C, D$ are to be determined. Using the continuity conditions
\[
  \begin{cases}
  g(x,y)|_{x=y^{+}} = g(x,y)|_{x=y^{-}},\\
  \frac{{\rm d}g(x,y)}{{\rm d}x}|_{x=y^{+}} - \frac{{\rm d}g(x,y)}{{\rm
d}x}|_{x=y^{-}} = -1,\\
  g(x,y)|_{x=0^{+}} = g(x,y)|_{x=0^{-}},\\
  \frac{{\rm d}g(x,y)}{{\rm d}x}|_{x=0^{+}} = \frac{{\rm d}g(x,y)}{{\rm
d}x}|_{x=0^{-}},
 \end{cases}
\]
we get a linear system:
\[
  \begin{cases}
  Ae^{{\rm i}\kappa_1 y} = Be^{{\rm i}\kappa_1 y} + Ce^{-{\rm i}\kappa_1 y},\\
  {\rm i}\kappa_1 A e^{{\rm i}\kappa_1 y} - {\rm i}\kappa_1 B e^{{\rm i}\kappa_1
y} + {\rm i}\kappa_1 C e^{-{\rm i}\kappa_1 y} = -1, \\
  B + C = D,\\
  {\rm i}\kappa_1 B - {\rm i}\kappa_1 C = -{\rm i}\kappa_2 D.
 \end{cases}
\]
A simple calculation yields that 
\[
  \begin{cases}
  A = {\rm i} \frac{\kappa_1 - \kappa_2}{2\kappa_1(\kappa_1 + \kappa_2)}e^{{\rm
i}\kappa_1 y} + \frac{\rm i}{2\kappa_1}e^{-{\rm i}\kappa_1 y},\\
  B = {\rm i} \frac{\kappa_1 - \kappa_2}{2\kappa_1(\kappa_1 + \kappa_2)}e^{{\rm
i}\kappa_1 y},\\
  C = \frac{\rm i}{2\kappa_1}e^{{\rm i}\kappa_1 y},\\
  D = \frac{\rm i}{\kappa_1 + \kappa_2}e^{{\rm i}\kappa_1 y},
  \end{cases}
\]
which gives
\[
 g(x, y)=\begin{cases}
           {\rm i} \frac{\kappa_1 - \kappa_2}{2\kappa_1(\kappa_1 +
\kappa_2)}e^{{\rm i}\kappa_1 (x+y)} + \frac{\rm i}{2\kappa_1}e^{{\rm i}\kappa_1
|x-y|}, & x>0,\\
 \frac{\rm i}{\kappa_1 + \kappa_2}e^{{\rm i}(\kappa_1 y-\kappa_2 x)}, & x<0.
         \end{cases}
\]

If $y<0$, the solution has the following form
\[
 g(x,y) = \begin{cases}
  Ae^{-{\rm i}\kappa_2 x}, & x<y,\\
  Be^{-{\rm i}\kappa_2 x} + Ce^{{\rm i}\kappa_2 x} , & y<x<0,\\
  De^{{\rm i}\kappa_1 x}, & x>0.
 \end{cases}
\]
Using the continuity conditions
\[
  \begin{cases}
  g(x,y)|_{x=y^{+}} = g(x,y)|_{x=y^{-}},\\
  \frac{{\rm d}g(x,y)}{{\rm d}x}|_{x=y^{+}} - \frac{{\rm d}g(x,y)}{{\rm
d}x}|_{x=y^{-}} = -1,\\
  g(x,y)|_{x=0^{+}} = g(x,y)|_{x=0^{-}},\\
  \frac{{\rm d}g(x,y)}{{\rm d}x}|_{x=0^{+}} = \frac{{\rm d}g(x,y)}{{\rm
d}x}|_{x=0^{-}},
 \end{cases}
\]
we obtain 
\[
  \begin{cases}
  Ae^{-{\rm i}\kappa_2 y} = Be^{-{\rm i}\kappa_2 y} + Ce^{{\rm i}\kappa_2 y},\\
  -{\rm i}\kappa_2 A e^{-{\rm i}\kappa_2 y} + {\rm i}\kappa_2 B e^{-{\rm
i}\kappa_2 y} - {\rm i}\kappa_2 C e^{{\rm i}\kappa_2 y} = 1, \\
  B + C = D,\\
  -{\rm i}\kappa_2 B + {\rm i}\kappa_2 C = {\rm i}\kappa_1 D.
 \end{cases}
\]
It follows from solving the above linear system that 
\[
  \begin{cases}
  A = {\rm i} \frac{\kappa_2 - \kappa_1}{2\kappa_2(\kappa_1 + \kappa_2)}e^{-{\rm
i}\kappa_2 y} + \frac{\rm i}{2\kappa_2}e^{{\rm i}\kappa_2 y},\\
  B = {\rm i} \frac{\kappa_2 - \kappa_1}{2\kappa_2(\kappa_1 + \kappa_2)}e^{-{\rm
i}\kappa_2 y},\\
  C = \frac{\rm i}{2\kappa_2}e^{-{\rm i}\kappa_2 y},\\
  D = \frac{\rm i}{\kappa_1 + \kappa_2}e^{-{\rm i}\kappa_2 y},
  \end{cases}
\]
which yields 
\[
 g(x,y)=\begin{cases}
  {\rm i} \frac{\kappa_2 - \kappa_1}{2\kappa_2(\kappa_1 + \kappa_2)}e^{-{\rm
i}\kappa_2 (x+y)} + \frac{\rm i}{2\kappa_2}e^{{\rm i}\kappa_2 |x-y|}, &
x<0,\\[10pt]
 \frac{\rm i}{\kappa_1 + \kappa_2}e^{{\rm i}(- \kappa_2 y + \kappa_1 x )}, &
x>0.
 \end{cases}
\]

\end{document}